\RequirePackage[l2tabu, orthodox]{nag}
\documentclass[a4paper,11pt]{amsart}
\usepackage[margin=32mm,bottom=40mm]{geometry}

\usepackage{lmodern}
\usepackage[stretch=10]{microtype}
\usepackage[T1]{fontenc}
\usepackage{tikz, tikz-cd, stmaryrd, amsmath, amsthm, amssymb,
	hyperref, comment, todonotes, bbm, mathrsfs, txfonts, pythonhighlight}
\usepackage[shortlabels]{enumitem}
\usetikzlibrary{arrows}
\usetikzlibrary{positioning}
\usepackage[utf8x]{inputenc}
\newcommand{\bb}{\textbf}

\newcommand{\ol}{\overline}
\newcommand{\mc}{\mathcal}
\newcommand{\wh}{\widehat}

\newcommand{\mf}{\mathfrak}
\newcommand{\ms}{\mathscr}
\renewcommand{\AA}{\mathbb{A}}
\renewcommand{\SS}{\mathbb{S}}

\newcommand{\ZZ}{\mathbb{Z}}

\newcommand{\GG}{\mathbb{G}}

\newcommand{\QQ}{\mathbb{Q}}

\newcommand{\NN}{\mathbb{N}}
\newcommand{\FF}{\mathbb{F}}

\DeclareMathOperator{\diag}{diag}

\DeclareMathOperator{\im}{im}

\DeclareMathOperator{\Span}{Span}

\DeclareMathOperator{\Gl}{Gl}

\DeclareMathOperator{\GCD}{GCD}
\DeclareMathOperator{\LCM}{LCM}
\DeclareMathOperator{\ord}{ord}

\DeclareMathOperator{\Gal}{Gal}

\DeclareMathOperator{\End}{End}

\DeclareMathOperator{\Jac}{Jac}

\DeclareMathOperator{\GSp}{GSp}
\DeclareMathOperator{\Sp}{Sp}
\DeclareMathOperator{\PSp}{PSp}

\DeclareMathOperator{\divv}{div}
\DeclareMathOperator{\Pic}{Pic}

\theoremstyle{plain}
\newtheorem{Theorem}{Theorem}[section]

\newtheorem{Lemma}[Theorem]{Lemma}
\newtheorem{Corollary}[Theorem]{Corollary}

\newtheorem{Proposition}[Theorem]{Proposition}

\newtheorem{Question}[Theorem]{Question}
\newtheorem{Setup}[Theorem]{Setup}

\theoremstyle{definition}
\newtheorem{Definition}[Theorem]{Definition}
\newtheorem{Example}[Theorem]{Example}

\theoremstyle{remark}

\numberwithin{equation}{section}
\allowdisplaybreaks
\begin{document}
	
	\title[Products of hyperelliptic Jacobians with maximal Galois image]{Products of hyperelliptic Jacobians\\ with maximal Galois image}
	\author[J. Garnek]{J\k{e}drzej Garnek}
	\address{Faculty of Mathematics and Computer Science\\
	Uniwersytetu Pozna\'{n}skiego~4, 61-614 Poznan, Poland}
	\email{jgarnek@amu.edu.pl}
	\subjclass[2020]{Primary 14H40, Secondary 11F80}
	\keywords{}
	\date{}
	
	\begin{abstract}
		In this note we study the associated adelic representation of a product of hyperelliptic Jacobians.
		We give a simple criterion that ensures that this representation
		has maximal Galois image in a certain sense. As an application, we provide a method of constructing
		products of Jacobians with division fields as disjoint as they can be.
	\end{abstract}
	
	\maketitle
	\bibliographystyle{plain}
	
	\section{Introduction}
	Let $A$ be a principally polarized abelian variety of dimension $g$ over $\QQ$. The Galois action on the torsion points of $A$ gives rise to $\ell$-adic representations $\rho_{A, \ell} : \Gal(\ol{\QQ}/\QQ) \to \GSp_{2g}(\ZZ_{\ell})$
	for every prime $\ell$ and an adelic representation
	\[
		\rho_A :  \Gal(\ol{\QQ}/\QQ) \to \GSp_{2g}(\wh{\ZZ}).
	\]
	Studying the images of $\rho_A$ and $\rho_{A, \ell}$ is a classical and well-studied topic.
	A classical result of Serre (cf. \cite[Theorem 3]{SerreOeurvesIV}) states that if $g \in \{ 2, 6 \}$ or $2 \nmid g$ and $\End_{\ol{\QQ}}(A) = \ZZ$, then 
	the image of $\rho_A$ is open in $\GSp_{2g}(\wh{\ZZ})$. In particular, $\rho_{A, \ell}$ is surjective for
	sufficiently large primes $\ell$. In general, the behaviour of $\rho_{A, \ell}$ is predicted by the Mumford--Tate conjecture (see \cite{Farfan_Survey_HTMTC} for a survey).\\
	
	It seems natural to ask whether $\rho_A$ can be surjective. This is possible if $g \ge 3$, cf.
	\cite[Corollary~1.3]{Landesman_Lombardo_Surjectivity_Rational_Families}. 
	However, for many classes of abelian varieties there are some limitations
	on the size of the image of $\rho_A$. For example,
	if $A$ is a product of $n$ elliptic curves, then $\im \rho_A \subset \Gl^{\Delta}(\wh{\ZZ}) := \{ (A_1, \ldots, A_n) \in \prod_{i = 1}^n \Gl_2(\wh{\ZZ})
	: \det A_1 = \ldots = \det A_n \}$ and moreover:
	\begin{equation} \label{eqn:maximal_EC}
		[\Gl^{\Delta}(\wh{\ZZ}) : \im \rho_A] \ge 2^n
	\end{equation}
	(this was observed by Serre, cf. \cite[Proposition~22]{Serre_proprietes}). In particular, $\rho_A$ in this case cannot be surjective. If equality in~\eqref{eqn:maximal_EC} holds for a product $A$ of elliptic curves, we say that $A$ has \emph{maximal Galois image}. Jones' results show that almost all products of elliptic curves over $\QQ$ have maximal Galois image, cf.~\cite{Jones_Almost_all_ECs_are_Serre} and \cite{Jones_Pairs_of_ECs}. The article~\cite{Landesman_Lombardo_Surjectivity_Rational_Families} 
	generalizes the notion of maximal Galois image to other families, and shows that almost all
	abelian varieties in a given family have maximal Galois image under some mild assumptions.
	We recall what it means to have maximal Galois image in the case of a product of hyperelliptic Jacobians in Section~\ref{sec:max_gal_image}, cf. Definition~\ref{def:max_gal_image}.
	It turns out that a generic hyperelliptic curve has maximal Galois image, cf. \cite{Landesman_Hyperelliptic_curves_maximal}.
	In particular, this implies that a product of generic hyperelliptic Jacobians has maximal Galois image. However, those results do not give an explicit method of constructing explicit products of hyperelliptic Jacobians with maximal Galois image. Such examples were constructed in the article~\cite{Anni_Dokchitser_Constructing_hyperelliptic}.
	The article \cite{Daniels_Hatley_Ricci_Elliptic_curves} yields a criterion for a product
	of elliptic curves to have maximal Galois image and also provides an explicit example.\\
	
	In this article we investigate the adelic representation of an abelian variety $J$, which is a product of hyperelliptic Jacobians
	and give an explicit criterion for the image of $\rho_J$ to be maximal in the sense of \cite{Landesman_Lombardo_Surjectivity_Rational_Families}.
	We use the following setup.
	\begin{Setup} \label{setup} 
		Let $f_i \in \ZZ[x]$ be separable polynomials of degrees $2g_i + 2$ and of discriminants $\Delta_i$ for $i = 1, \ldots, n$.
		Let $J_i$ be the Jacobian of the hyperelliptic curve with the affine part given by $y^2 = f_i(x)$ and denote
		$J := J_1 \times \ldots \times J_n$. Finally, let $g := g_1 + \ldots + g_n$ be the dimension of $J$.
	\end{Setup}
	The following is the main result of the paper.
	\begin{Theorem} \label{thm:main_theorem}
		Keep Setup~\ref{setup} and assume that $g_1, \ldots, g_n \ge 2$. Suppose also that:
		\begin{enumerate}
			\item $J_1, \ldots, J_n$ have maximal Galois image,
			
			\item there exist primes $\ell_1, \ldots, \ell_n > 2$ such that $\ord_{\ell_i}(\Delta_i) = 1$ for $i = 1, \ldots, n$ and $\ord_{\ell_i}(\Delta_j) = 0$ for $i \neq j$.
		\end{enumerate}
		Then $J$ has maximal Galois image.
	\end{Theorem}
	Note that for a generic hyperelliptic curve over $\QQ$ there exists a prime $\ell$ such that $\ord_{\ell}(\Delta) = 1$, where $\Delta$ is the discriminant of the defining polynomial (cf. \cite[Theorem~29]{van_bommel_2014}). Moreover, if $J = J_1 \times \ldots \times J_n$ has maximal  Galois image, then each Jacobian $J_i$ also has maximal Galois image. Therefore the assumptions
	of Theorem~\ref{thm:main_theorem} are quite reasonable.\\
	
	\noindent In order to illustrate our result we offer two examples.
	\begin{Example} \label{ex:three_jacobians}
		Let $J_i$ be the Jacobian of the hyperelliptic curve with the affine equation $y^2 = f_i(x)$ for $i = 1, 2, 3$, where
		\begin{align*}
			f_1(x) &:= x^6 + 7471225 \cdot x^5 + 16548721 \cdot x^4 + 6639451 \cdot x^3\\
			&+ 16857421 \cdot x^2 + 20754195 \cdot x + 9508695,\\
			f_2(x) &:= x^8 + 10781051650 \cdot x^7 + 5302830080 \cdot x^6 + 33362176 \cdot x^5 + 10656581376 \cdot x^4\\
			&+ 5522318080 \cdot x^3 + 4238752256 \cdot x^2 + 3613465600 \cdot x + 3725404480,\\
			f_3(x) &:= x^{14} +1122976550518058592759939074 \cdot x^{13} +
			10247323490706358348644352 \cdot x^{12}\\
			&+1120184609916242124087443456 \cdot x^{11} + 186398290364786000921886720 \cdot x^{10}\\
			&+1685990245699349559300014080 \cdot x^9 + 387529952672653585935499264 \cdot x^8\\
			&+1422826957983635547417870336 \cdot x^7 + 585983998625429997308035072 \cdot x^6\\
			&+ 607434202225985243206107136 \cdot x^5 +1820210247550502007557029888 \cdot x^4\\
			&+ 533014336994715937945092096 \cdot x^3 + 595803405154942945879752704 \cdot x^2\\
			&+1276845913825955586899050496 \cdot x +1323672381818030813822668800.
		\end{align*}
		Note that $J_1$, $J_2$ and $J_3$ are abelian varieties of dimensions $2$, $3$
		and $6$ respectively (see Section~\ref{sec:max_gal_image} for the prerequisites on hyperelliptic curves). $J_1$ and $J_2$ have maximal Galois image by \cite[Theorem 1.3]{Landesman_Hyperelliptic_curves_maximal}.
		$J_3$ has maximal Galois image by \cite[\S 8]{Anni_Dokchitser_Constructing_hyperelliptic} combined with
		\cite[Lemma 4.3]{Landesman_Hyperelliptic_curves_maximal} and \cite[Theorem 1.1(a)]{Yelton_Lifting_images} (see the beginning of Section~\ref{sec:example} for a detailed explanation).
		Let $\ell_1 = 421$, $\ell_2 = 13$, $\ell_3 = 89$. Then $\ell_i | \Delta_i$, $\ell_i^2 \nmid \Delta_i$ for $i = 1, 2, 3$ and $\ell_i \nmid \Delta_j$ for $i \neq j$. Thus the product of Jacobians $J := J_1 \times J_2 \times J_3$ has maximal Galois image by Theorem~\ref{thm:main_theorem}.
	\end{Example}
	\begin{Example} \label{ex:family}
		Let for any $t \in \ZZ$, $F_t(x) := f_3(x) + N \cdot t$, where $f_3$ is as in Example~\ref{ex:three_jacobians} 
		and $N :=  2201590757511816436065484800$. Let $J_t$ denote the Jacobian of the curve with the affine equation
		$y^2 = F_t(x)$. Then there exists an infinite increasing sequence $\{t_i\}_{i = 1}^{\infty}$ of integers such that
		for every $n$ the abelian variety $J_{t_1} \times \ldots \times J_{t_n}$ has maximal Galois image.
		One can take for example $(t_1, \ldots, t_{10}) := (0, 2, 5, 7, 9, 10, 13, 14, 17, 21)$. See Section~\ref{sec:example} for the details.
	\end{Example}
	The primary motivation for this article was the following question.
	\begin{Question}
		Suppose that $A_1$, $A_2$ are abelian varieties over $\QQ$. What can be said about the intersection of division fields
		$\QQ(A_1[n]) \cap \QQ(A_2[n])$?
	\end{Question}
	Several variants of this question were studied in literature, see e.g. \cite{Kowalski_Some_local_global}, \cite{Lombardo_Non-isogenous}, \cite{Daniels_Lozano_Conicidences}, \cite{Daniels_Hatley_Ricci_Elliptic_curves}. The philosophy
	is that if $A_1$ and $A_2$ do not have isogenous factors, then their division fields are \emph{usually}
	as disjoint as they can be. As an application of our result we show that if $J_1, \ldots, J_n$ satisfy the assumptions of Theorem~\ref{thm:main_theorem}, then their division fields are maximally disjoint. This is to be expected,
	but nevertheless we have not found this result in the literature.
	
	For any non-empty $\ms A \subset \{ 1, \ldots, n \}$, let 
	\[
		J_{\ms A} := \prod_{a \in \ms A} J_a.
	\]
	\begin{Theorem} \label{thm:division_fields}
		Keep the assumptions of Theorem~\ref{thm:main_theorem}.
		Let $\ms A, \ms B$ be disjoint non-empty subsets of $\{1, \ldots, n\}$.
		Then for any $m_1, m_2 \in \NN$ the intersection $\QQ(J_{\ms A}[m_1]) \cap \QQ(J_{\ms B}[m_2])$ equals
		\begin{equation} \label{eqn:intersection_abelian}
			\QQ(\zeta_{m_1}) \left(\sqrt{\Delta_i^{m_1 + 1}} : i \in \ms A \right) \cap \QQ(\zeta_{m_2}) \left(\sqrt{\Delta_i^{m_2 + 1}} : i \in \ms B \right).
		\end{equation}
	\end{Theorem}
	Note that the intersection~\eqref{eqn:intersection_abelian} can be easily computed for any given $m_1, m_2, \Delta_1, \ldots, \Delta_n$.
	In particular, if $\GCD(m_1, \Delta_1 \cdot \ldots \cdot \Delta_n) = \GCD(m_2, \Delta_1 \cdot \ldots \cdot \Delta_n) = 1$,
	then for any disjoint non-empty subsets $\ms A$, $\ms B$ of $\{1, \ldots, n\}$
	\[
	\QQ(J_{\ms A}[m_1]) \cap \QQ(J_{\ms B}[m_2]) = \QQ(\zeta_d),
	\]
	where $d := \GCD(m_1, m_2)$.\\
	
	The paper is structured as follows. Section~\ref{sec:max_gal_image} contains a brief review of needed facts. In Section~\ref{sec:criterion} we show
	that proving that $J$ has maximal Galois image comes down to checking whether $2$-division fields
	are disjoint and $\ell$-division fields are pairwise disjoint for every prime $\ell > 2$. This follows by Jones' group theoretic results from~\cite{Jones_Gl2_reps}. In Section~\ref{sec:8-torsion} we show that the $2$-division fields are disjoint. In Section~\ref{sec:p-torsion} we investigate the $\ell$-division fields for primes $\ell > 2$. To this end we 
	analyse the reduction of the Jacobians and the ramification in the division fields. Section~\ref{sec:corollary} is devoted to the proof of Theorem~\ref{thm:division_fields}. In the final section we discuss Example~\ref{ex:family}.\\
	
	It seems hard to get rid of the condition that $g_1, \ldots, g_n \neq 1$ in Theorem~\ref{thm:main_theorem}.
	This would require checking whether $36$-division fields are disjoint, cf. \cite{Daniels_Hatley_Ricci_Elliptic_curves}.
	This is plausible in the case when $g_1 = \ldots = g_n = 1$,
	since there are explicit descriptions of $m$-division fields of elliptic curves for small $m$.
	However, it is much harder to describe e.g. the $3$-division field of a general hyperelliptic curve.
	
	\subsection*{Notation}
	We denote by $\lambda : \GSp_{2g} \to \GG_m$ the similitude character.
	For any group scheme $\mc G$ over $\ZZ$ and $m \in \ZZ_+$ we consider the natural reduction homomorphism $\Phi_{\mc G, m} : \mc G(\wh{\ZZ}) \to \mc G(\ZZ/m)$. Also, for any subgroup $H$ of $\mc G(\wh{\ZZ})$, we write $H(m) := \Phi_{\mc G, m}(H)$. Similarly, for any prime $\ell$ we denote by $H(\ell^{\infty})$ the image of $H$ under the projection $\mc G(\wh{\ZZ}) \to \mc G(\ZZ_{\ell})$.
	
	Consider the action of $S_{2g_i+2}$ on the vector space $(\ZZ/2)^{2g_i + 2}$ equipped with the standard inner product, i.e.
	\[
		\langle (x_1, \ldots, x_{2g_i + 2}), (y_1, \ldots, y_{2g_i + 2}) \rangle := x_1 \cdot y_1 + \ldots + x_{2g_i + 2} \cdot y_{2g_i + 2}.
	\]
	It descends to an action on $(1, \ldots, 1)^{\bot}/\Span_{\FF_2}\{(1, \ldots, 1)\} \cong (\ZZ/2)^{2g_i}$ and provides
	a map $S_{2g_i + 2} \to \GSp_{2g_i}(\ZZ/2)$. This map is injective as
	we are assuming that $g_i \ge 2$.
	We denote its image by $\SS_i$.
	Also, we write $\AA_i$ for the copy of the alternating group on $2g_i + 2$ elements inside of $\SS_i$.
	
	\subsection*{Acknowledgements}
	The author wishes to express his gratitude to Wojciech Gajda for suggesting this problem and for helpful
	conversations.  Additionally, the author extends thanks to Bartosz Naskrecki for discussions that lead to Example~\ref{ex:family}.
	The author also acknowledges Marc Hindry, Grzegorz Banaszak, Davide Lombardo and Nathan Jones
	for comments on the topic. Finally, the author is grateful for the valuable suggestions
	of two anonymous reviewers. In particular, the author would like to thank one of the reviewers and David Savitt for identifying a mistake in an earlier version of the manuscript. The author was supported by the grant 038/04/N\'{S}/0011, which is a part of the
	project "Initiative of Excellence -- Research University" on Adam Mickiewicz University.
	
	\section{Preliminaries} \label{sec:max_gal_image}
	In this section we briefly recall some facts needed in the sequel. We start by reviewing basic properties of hyperelliptic curves. For a reference see e.g.~\cite{Handbook_elliptic_hyperelliptic}. Let $K$ be a field of characteristic different then~$2$.
	The affine curve given by the equation $y^2 = f(x)$ turns out to be smooth for any separable polynomial $f \in K[x]$. The corresponding smooth projective curve $C$
	is called a \bb{hyperelliptic curve}.
	Its genus equals $\left\lfloor \frac{\deg f - 1}{2} \right\rfloor$. 
	Note that for any distinct roots $\alpha_1, \alpha_2$ of $f$:
	\[
		2 \cdot ((\alpha_1, 0) - (\alpha_2, 0)) = \divv \left(\frac{x - \alpha_1}{x - \alpha_2} \right)
	\]
	and thus $(\alpha_1, 0) - (\alpha_2, 0) \in \Pic^0(C)[2]$. This fact may be used to show that
	the $2$-torsion field of $\Jac(C)$, the Jacobian of $C$, is the splitting field of the polynomial $f(x)$.
	Suppose now that $K = \QQ$ and $f \in \ZZ[x]$ has discriminant~$\Delta$. Then 
	the equation $y^2 = \ol f(x)$ (where $\ol f \in \FF_{\ell}[x]$ is the reduction of
	$f$ modulo $\ell$) defines a smooth hyperelliptic curve over $\FF_{\ell}$ for all primes
	$\ell \nmid 2 \cdot \Delta$. Thus $\Jac(C)$ has good reduction over all such primes.\\
	
	The following two facts are standard. We give a proof for a lack of a reference.
	\begin{Lemma} \label{lem:one_double_root}
		Let $f \in \ZZ[x]$ be a polynomial with discriminant $\Delta$ and let $\ol f \in \FF_{\ell}[x]$ denote its reduction modulo a prime $\ell$.
		If $\ord_{\ell}(\Delta) = 1$, then $\ol f$ has a unique double root in $\ol{\FF}_{\ell}$.
	\end{Lemma}
	\begin{proof}
		The condition $\ell | \Delta$ implies that $\ol f$ has a double root $\alpha \in \FF_{\ell^m}$ for some $m \ge 1$. Let
		$\QQ_{\ell^m}$ denote the unique unramified extension of $\QQ_{\ell}$ of degree $m$ and let $\ZZ_{\ell^m}$ be its ring of integers.
		By using the map $x \mapsto x + \alpha$, we may assume that $\alpha = 0$ and $f \in \ZZ_{\ell^m}[x]$. Factor $\ol f$ as $x^{e} \cdot \ol h(x)$,
		where $\ol h \in \FF_{\ell^m}[x]$ and $\ol h(0) \neq 0$. We show now that $e = 2$ and that $\ol h$ is separable. By Hensel's lemma,
		the factorisation $\ol f = x^e \cdot \ol h$ lifts to a factorisation
		$f = g \cdot h$, where $g, h \in \ZZ_{\ell^m}[x]$, $g = x^e + \sum_{i = 0}^{e-1} a_i x^i$ and $\ord_{\ell}(a_i) > 0$ for $i = 0, \ldots, e-1$.
		Let $\Delta(g), \Delta(h)$ be the discriminants of $g$ and $h$. Then, by the formula for the discriminant of a product of polynomials (see e.g. \cite[Proposition~15.A.2]{Childs_Concrete_Higher_algebra}), $\Delta(g) \cdot \Delta(h) | \Delta(f)$, which
		easily implies that $\ord_{\ell}(\Delta(g)) = 1$ and $\ord_{\ell}(\Delta(h)) = 0$. In particular, $\ol h$ is separable.
		Moreover, $\Delta(g)$ is the determinant of the $(2e - 1) \times (2e - 1)$ Sylvester matrix
		\begin{align*}
			\Delta(g) =
			\left| {\begin{array}{cccccccc}
					1  & a_{e-1}  &\cdots  	&\cdots &a_{0}   	&0	  	& \cdots  & 0       \\
					0        & 1      &a_{e-1} 	&\cdots &\cdots  	&a_{0}  	& \cdots  & 0       \\
					\vdots & \ddots    &\dots       	&\dots  &\dots   	&\ddots 	& \ddots  & \vdots  \\
					\vdots & \ddots    &0            	&1  &\cdots  	&		& \cdots  & a_{0}   \\
					e & e \cdot a_{e-1} &\cdots  	&\cdots &1 \cdot a_{1}   	&0	  	& \cdots  & 0       \\
					0	     & e    &e \cdot a_{e-1} 	&\cdots &\cdots  	&1 \cdot a_{1}  	& \cdots  & 0        \\
					\vdots  & \ddots    &\dots   	&\dots  &\dots   	&\ddots 	& \ddots  & \vdots   \\
					0 	     & \ddots    &0       		&e &e \cdot a_{e-1} 	&\cdots 	& \cdots  & 1 \cdot a_1    \\
			\end{array} } \right|.
		\end{align*}
		Note that the last $(e-1)$ columns of this determinant are divisible by $p$. Therefore, by expanding the determinant
		one sees that $\ord_{\ell}(\Delta(g)) \ge e-1$. But $\ord_{\ell}(\Delta(g)) = 1$, which proves that $e = 2$. 
	\end{proof}
	\begin{Corollary} \label{cor:toric_rk_one}
		Let $C$ be a hyperelliptic curve with the affine part given by the equation $y^2 = f(x)$, where $f \in \ZZ[x]$
		is a polynomial with discriminant $\Delta$. Suppose that $\ell$ is a prime satisfying $\ord_{\ell}(\Delta) = 1$.
		Then $\Jac(C)$ has semistable reduction of toric rank one over $\QQ_{\ell}$.
	\end{Corollary}
	\begin{proof}
		By Lemma~\ref{lem:one_double_root} $\ol f(x)$ has a unique double root $\alpha \in \ol{\FF}_{\ell}$. Moreover, $\alpha \in \FF_{\ell}$, since
		$x - \alpha = \GCD(\ol f, \ol f')$. Thus the hyperelliptic curve over $\FF_{\ell}$ with the affine equation $y^2 = \ol f(x)$ has a single node. This implies that the equation
		$y^2 = f(x)$ gives a semistable model of the curve over $\ZZ_{\ell}$ with a node in the special fiber. The dual graph of the special fiber is a single vertex with a loop
		and thus $\Jac(C)$ has toric rank one over $\QQ_{\ell}$ by \cite[Example~8, p.~246]{BLR1990}.
	\end{proof}
	
	We explain now what it means for a product of hyperelliptic Jacobians to have maximal Galois image. Let $J$ be as in Setup~\ref{setup}. In the sequel we treat $\prod_{i=1}^n \GSp_{2g_i}$ as a subgroup of $\GSp_{2g}$,
	by identifying a tuple of matrices $(A_1, \ldots, A_n)$ with the block diagonal matrix $\diag(A_1, \ldots, A_n)$.
	There are several natural limitations on the size of the image of $\rho_J$. Firstly, $\im \rho_J \subset \GSp^{\Delta}(\wh{\ZZ})$, where $\GSp^{\Delta}$ is a subgroup scheme of $\GSp_{2g}$ given by
	\begin{equation} \label{eqn:gsp_def}
		\GSp^{\Delta}(R) := \left\{ (A_1, \ldots, A_n) \in \prod_{i=1}^n \GSp_{2g_i}(R) : \lambda(A_1) = \ldots = \lambda(A_n) \right\}
	\end{equation}
	for any ring $R$. This follows from the fact
	that the composition
	\[
	\lambda \circ \rho_{J_i} : \Gal(\ol{\QQ}/\QQ) \to \wh{\ZZ}^{\times}
	\]
	equals the cyclotomic character for $i = 1, \ldots, n$. 
	
	Secondly, since $\QQ(J[2])$ is the splitting field of the polynomial $f_1 \cdot \ldots \cdot f_n$,
	we have
	\[
	\Gal(\QQ(J[2])/\QQ) \le \SS_1 \times \ldots \times \SS_n,
	\]
	where $\SS_i \subset \Sp_{2g_i}(\ZZ/2)$ is a copy of the symmetric group on $\deg f_i$ elements. Therefore
	\[
	\im \rho_J \subset \Phi_{\GSp^{\Delta}, 2}^{-1}(\SS_1 \times \ldots \times \SS_n),
	\]
	where $\Phi_{\GSp^{\Delta}, 2} : \GSp^{\Delta}(\wh{\ZZ}) \to \GSp^{\Delta}(\ZZ/2)$ is the reduction map.
	We prove below a lower bound for the quotient of those two groups, using the containment $\QQ(\sqrt{\Delta_1}, \ldots, \sqrt{\Delta_n}) \subset \QQ(J[2]) \cap \QQ(\zeta_{\infty})$.
	\begin{Proposition} \label{prop:inequality_index_C1_C2}
		Keep the above notation. The following inequality holds
		\begin{equation*} 
			[\Phi_{\GSp^{\Delta}, 2}^{-1}(\SS_1 \times \ldots \times \SS_n) : \im \rho_J] \ge 2^n.
		\end{equation*}
	\end{Proposition}
	\begin{proof}
		Write $\Delta_i'$ for the squarefree part of $\Delta_i$ and let
		\[
		D_i :=
		\begin{cases}
			|\Delta_i'|, & \textrm{ if } \Delta_i' \equiv 1 \pmod 4,\\
			4 \cdot |\Delta_i'|, & \textrm{ otherwise,}
		\end{cases}
		\]
		and $M_i := \LCM(D_i, 2)$.
		It is a standard fact that $\QQ(\sqrt{\Delta_i}) \subset \QQ(\zeta_D)$ if and only if $D_i | D$.
		Denote $\mathscr S_i := \Phi^{-1}_{\GSp_{2g_i}, 2}(\SS_i)$.
		The Galois action on $\QQ(\sqrt{\Delta_i})$ coming from the containment in $\QQ(J_i[2])$ can be described as
		\begin{equation} \label{eqn:action_delta_J2}
			\sigma \left(\sqrt{\Delta_i'} \right) = \varepsilon_i(\rho_{J_i}(\sigma)) \cdot \sqrt{\Delta_i'} \quad
			\textrm{ for any $\sigma \in \Gal(\QQ(J_i[M_i])/\QQ)$,}
		\end{equation}
		where $\varepsilon_i$ denotes the composition
		\[
		\mathscr S_i(M_i) \to \mathscr S_i(2) = \SS_i \to \SS_i/\AA_i \cong \{ \pm 1\}.
		\]
		On the other hand, the Galois action coming from the inclusion $\QQ(\sqrt{\Delta_i}) \subset \QQ(\zeta_{M_i})$
		is given by
		\begin{equation} \label{eqn:action_delta_cyclotomic}
			\sigma\left(\sqrt{\Delta_i'} \right) = \left( \frac{\Delta_i'}{\lambda(\rho_{J_i}(\sigma))} \right) \cdot \sqrt{\Delta_i'} \quad \textrm{ for any $\sigma \in \Gal(\QQ(J_i[M_i])/\QQ)$,}
		\end{equation}
		where $\left(\frac{\Delta_i'}{\cdot}\right)$ is the Kronecker symbol.
		Define the Serre subgroup associated to $J_i$ as
		\[
		H_i := \ker \left( \left( \frac{\Delta_i'}{\lambda(\cdot) } \right) \cdot \varepsilon_i(\cdot) : \mathscr S_i(M_i) \to \{ \pm 1 \} \right).
		\]
		Note that $H_i$ is an index $2$ subgroup of $\ms S_i(M_i)$.
		Then~\eqref{eqn:action_delta_J2} and~\eqref{eqn:action_delta_cyclotomic} show that
		\begin{align}
			\im \rho_{J_i} &\subset \Phi_{\GSp_{2g_i}, M_i}^{-1}(H_i), \label{eqn:containment_J_i}\\
			\im \rho_J &\subset \left(\prod_{i = 1}^n \Phi_{\GSp_{2g_i}, M_i}^{-1}(H_i) \right) \cap \GSp^{\Delta}(\wh{\ZZ}). \label{eqn:containment_J}
		\end{align}
		In particular, this yields the desired inequality, since
		\[
		\left[\Phi_{\GSp^{\Delta}, 2}^{-1}(\SS_1 \times \ldots \times \SS_n) : \left(\prod_{i = 1}^n \Phi_{\GSp_{2g_i}, M_i}^{-1}(H_i) \right) \cap \GSp^{\Delta}(\wh{\ZZ}) \right]
		= \prod_{i = 1}^n [\mathscr S_i(M_i) : H_i] = 2^n.
		\]
		See~\cite[Proposition~22]{Serre_proprietes} and \cite[\S 4]{Jones_Almost_all_ECs_are_Serre} for a similar reasoning for elliptic curves.
	\end{proof}
	\begin{Definition} \label{def:max_gal_image}
		We say that a product of hyperelliptic Jacobians $J$ defined over $\QQ$
		has \emph{maximal Galois image} if
		equality in Proposition~\ref{prop:inequality_index_C1_C2} holds.
	\end{Definition}
	Note that in the proof of the inequality of Proposition~\ref{prop:inequality_index_C1_C2}
	we used the Kronecker--Weber theorem. Hence this inequality 
	generally only holds when the underlying hyperelliptic curves are defined over $\QQ$.
	
	If $J$ has maximal Galois image, then \eqref{eqn:containment_J_i} and \eqref{eqn:containment_J} become equalities.
	Note that results of~\cite{Landesman_Lombardo_Surjectivity_Rational_Families} imply that
	a product of generic hyperelliptic Jacobians has maximal Galois image.\\
	
	In the sequel we will need the following result.
	\begin{Lemma} \label{lem:image_on_nondiv_by_Mi}
		Keep Setup~\ref{setup}. If $J_i$ has maximal Galois image, then for any integer $m$ not divisible by $M_1, \ldots, M_n$
		one has
		\[
		\Gal(J_i[m]/\QQ) \cong \mathscr S_i(m).
		\]
	\end{Lemma}
	\begin{proof}
		One easily checks that $H_i$ surjects onto $\mathscr S_i(d)$ for any $d | M_i$, $d \neq M_i$.
		Suppose now that $B_1 \in \mathscr S_i(m)$. Let $d := \GCD(M_i, m)$. Then $d \neq M_i$ and
		thus by the previous observation, there exists a matrix $B_2 \in H_i$ such that
		$B_1 \equiv B_2 \pmod d$. Therefore by the Chinese Remainder Theorem there exists a matrix $X \in \GSp_{2g_i}(\wh{\ZZ})$
		such that
		\[	\left\{
		\begin{array}{cccc}
			X &\equiv& B_1 &\pmod{m}\\
			X &\equiv& B_2 &\pmod{M_i}.
		\end{array}
		\right.
		\]
		Thus $X \in \Phi_{\GSp_{2g_i}, M_i}^{-1}(H_i)$ and the image of $X$ in $\GSp_{2g_i}(\ZZ/m)$ is $B_1$.
		This shows that $\Phi_{\GSp_{2g_i}, M_i}^{-1}(H_i)$ surjects onto $\mathscr S_i(m)$ and that
		\[
		\Gal(J_i[m]/\QQ) = \Phi_{\GSp_{2g_i}, M_i}^{-1}(H_i)(m) = \mathscr S_i(m). \qedhere
		\]
	\end{proof}

	\section{Criterion for maximality} \label{sec:criterion}
	In this section we prove a criterion for maximality of $\im \rho_J$, based on results of Jones from~\cite{Jones_Gl2_reps}.
	Before stating the result we recall two related notions from field theory.
	Let $L_1, \ldots, L_n$ be a family of extensions of a field $K$ contained in some common extension
	of~$K$. The fields $L_1, \ldots, L_n$ are said to be \bb{pairwise linearly
		disjoint over~$K$}, if $L_i \cap L_j = K$ for every $1 \le i < j \le n$. The fields
	$L_1, \ldots, L_n$ are \bb{linearly disjoint over $K$} if $\Gal(L_1 \cdot \ldots \cdot L_n/K) \cong \Gal(L_1/K) \times \ldots \times \Gal(L_n/K)$.
	This is equivalent to the condition
	\[
	L_i \cap \prod_{j \neq i} L_j = K \qquad \textrm{ for every } i = 1, \ldots, n.
	\]
	Thus linearly disjoint fields are pairwise linearly disjoint, but not vice versa. For example,
	$L_1 = \QQ(\sqrt 2)$, $L_2 = \QQ(\sqrt 3)$, $L_3 = \QQ(\sqrt 6)$ are pairwise linearly disjoint over~$\QQ$,
	but not linearly disjoint.
	\begin{Proposition} \label{prop:criterion_based_on_jones}
	Keep Setup~\ref{setup}. Suppose that $J_1, \ldots, J_n$ have maximal Galois image and that:
	\begin{itemize}
		\item the fields $\QQ(J_1[2]), \ldots, \QQ(J_n[2])$ are linearly disjoint over $\QQ$,
		\item for every prime $\ell > 2$, the fields $\QQ(J_1[\ell]), \ldots, \QQ(J_n[\ell])$ 
		are pairwise linearly disjoint over $\QQ(\zeta_{\ell})$.
	\end{itemize}
	Then $J$ has maximal Galois image.
	\end{Proposition}
	Before the proof we introduce the necessary notation and prove several auxiliary results.
	Let $\mc G$ be a subgroup scheme of $\Gl_r$ over $\ZZ$ with a homomorphism $\delta : \mc G \to \Gl_1$. Denote $\mc S := \ker \delta$. For any subgroup $H$ of $\mc G(\wh{\ZZ})$ (respectively of $\mc G(\ZZ/m)$) write $\mc SH := \mc S(\wh{\ZZ}) \cap H$ (respectively $\mc SH := \mc S(\ZZ/m) \cap H$). Define the sets $\mf g_{\ell^m}, \mf s_{\ell^m} \subset M_r(\ZZ_{\ell})$ by equalities
	\begin{align*}
		I + \ell^m \mf g_{\ell^m} = \ker(\mc G(\ZZ_{\ell}) \to \mc G(\ZZ/\ell^m)),\\
		I + \ell^m \mf s_{\ell^m} = \ker(\mc S(\ZZ_{\ell}) \to \mc S(\ZZ/\ell^m))
	\end{align*}
	for any prime $\ell$ and $m \in \NN$. Finally, let $\mf g_{\ell^m}(\ell), \mf s_{\ell^m}(\ell) \subset M_r(\ZZ/\ell)$ be the reductions $\bmod{\, \ell}$ of $\mf g_{\ell^m}$ and $\mf s_{\ell^m}$.
	
	The following lemma is standard. We recall its proof for completeness.
	\begin{Lemma} \label{lem:lower_bound_index}
		Keep the above assumptions. Suppose that $G$ is a subgroup of $\mc G(\wh{\ZZ})$ and $H$ is a subgroup of $G$. If the following conditions are satisfied:
		\begin{itemize}
			\item $\delta : H \to \wh{\ZZ}^{\times}$ is surjective,
			\item $[H, H] = \mc S H$,
		\end{itemize}
		then
		\[
			[G : H] \ge [\mc SG : [G, G]].
		\]
		Moreover, equality holds if and only if $[G, G] = [H, H]$.
	\end{Lemma}
	\begin{proof}
		From the diagram
		\begin{center}
			\begin{tikzcd}
			0 \arrow[r] & \mc S H \arrow[r] \arrow[d, hook] & H \arrow[r, "\delta"] \arrow[d, hook] & \wh{\ZZ}^{\times} \arrow[r] \arrow[d, Rightarrow, no head] & 0 \\
			0 \arrow[r] & \mc S G \arrow[r]                 & G \arrow[r, "\delta"]                 & \wh{\ZZ}^{\times} \arrow[r]                                & 0
			\end{tikzcd}
		\end{center}
		we see that $[G:H] = [\mc S G : \mc S H]$. But $\mc S H = [H, H] \le [G, G]$. This ends the proof.
	\end{proof}
	Keep the above notation. Let $\mc L$ be a finite set of prime numbers containing $2$.
	Consider the following assumptions:
		\begin{enumerate}
		\item[(A0)] $\mf s_{\ell^m}(\ell) = \mf s_{\ell}(\ell)$ for any prime $\ell$ and any $m \ge 1$.
		
		\item[(A1)] For any prime $\ell \not \in \mc L$ there
		exists a finite set of finite simple non-abelian groups $\mc{PS}_i(\ell)$ and surjective
		homomorphisms $\varpi_i : \mc S(\ZZ/\ell) \to \mc{PS}_i(\ell)$
		satisfying the conditions:
			\begin{itemize}
				\item for every normal subgroup $H \unlhd \mc S(\ZZ/\ell)$,
				either $H \subset \ker \varpi_i$ for some $i$, or $H = \mc S(\ZZ/\ell)$,
				
				\item for any primes $\ell' \neq \ell$, $\ell', \ell \not \in \mc L$ and any $i$,
				$\mc{PS}_i(\ell)$ is not a composition factor of $\mc S(\ZZ/\ell')$.
			\end{itemize}
			\item[(A2)] For every prime $\ell$, $\langle CD - DC : C, D \in \mf g_{\ell}(\ell) \rangle = \mf s_{\ell}(\ell)$.
			
			\item[(A3)] For every prime $\ell$, $\mf s_{\ell}(\ell)$ is generated as a $\ZZ/\ell$-vector space by a set of matrices $\{ u_i \}$, which satisfy $u_i^2 = 0$ and $I + u_i \in \mc S(\ZZ/\ell)$.		 
		\end{enumerate}	
		If we take $(\mc G, \delta) := (\GSp_{2g}, \lambda)$ then $\mc S = \Sp_{2g}$. In this context we denote $\mf{gsp}_{2g, \ell^m} := \mf g_{\ell^m}$ and $\mf{sp}_{2g, \ell^m} := \mf s_{\ell^m}$.
		Note that
		\begin{align*}
			\mf{gsp}_{2g, \ell}(\ell) &= \{ A \in M_{2g}(\ZZ/\ell) : A^T \Omega_g + \Omega_g A = (\lambdabar(A) - 1) \cdot \Omega_g \, 
			\, \textrm{ for some } \lambdabar(A) \in (\ZZ/\ell)^{\times} \},\\
			\mf{sp}_{2g, \ell}(\ell) &= \{ A \in \mf{gsp}_{2g, \ell}(\ell) : \lambdabar(A) = 1 \},
		\end{align*}
		where $\Omega_g := \begin{bmatrix}0&I_g\\-I_g&0\\\end{bmatrix}$ is the standard $2g \times 2g$ symplectic matrix.
	\begin{Lemma} \label{lem:A0-A3}
		The algebraic group $\mc G := \GSp^{\Delta}$ (defined by~\eqref{eqn:gsp_def})
		with the homomorphism $\delta : \mc G(\wh{\ZZ}) \to \wh{\ZZ}^{\times}$ given by
		\[
			\delta((A_1, \ldots, A_n)) := \lambda(A_1) = \lambda(A_2) = \ldots = \lambda(A_n)
		\]
		satisfies the assumptions (A0)-(A3) with $\mc L = \{ 2 \}$.
	\end{Lemma}
	\begin{proof}
		Note that:
		\begin{align*}
			\mc S &= \prod_{i = 1}^n \Sp_{2 g_i},\\
			\mf g_{\ell}(\ell) &= \left\{ (A_1, \ldots, A_n) \in \prod_{i = 1}^n \mf{gsp}_{2g_i, \ell}(\ell) : \lambdabar(A_1) = \ldots = \lambdabar(A_n) \right\},\\
			\mf s_{\ell}(\ell) &= \mf{sp}_{2g_1, \ell}(\ell) \times \ldots \times \mf{sp}_{2g_n, \ell}(\ell).
		\end{align*}
		\emph{Proof of (A2):} The proof of \cite[Lemma~6.2]{Jones_Gl2_reps} shows that for any prime $\ell$
		\[
		\mf{sp}_{2g, \ell}(\ell) = \langle AB - BA : A \in \mf{gsp}_{2g, \ell}(\ell), B \in \mf{sp}_{2g, \ell}(\ell) \rangle.
		\]
		Note that for any $A \in \mf{gsp}_{2g_i, \ell}(\ell)$ we can find $A_1 \in \mf{gsp}_{2g_1, \ell}(\ell), \ldots, A_n \in \mf{gsp}_{2g_n, \ell}(\ell)$, with $\lambdabar(A_1) = \lambdabar(A_2) = \ldots = \lambdabar(A_n)$ and $A_i = A$.
		Moreover for any $B \in \mf{sp}_{2g_i, \ell}(\ell)$
		\[
			[(A_1, A_2, \ldots, A_n), (0, 0, \ldots, B, \ldots, 0)] = (0, 0, \ldots, [A, B], \ldots, 0).
		\] 
		This ends the proof of (A2).\\
		\emph{Proof of (A1):} Let $\PSp_{2g}$ denote the projective symplectic group, i.e. the quotient of the symplectic group $\Sp_{2g}$ by the scalar matrices in the group. Recall that $\PSp_{2g}(\ZZ/\ell)$ is simple for $\ell > 2$ and $g > 1$,
		cf. \cite[Theorem~3.4.1]{OMeara_Symplectic_groups}.  For $i = 1, \ldots, n$, let
		\[
			\varpi_i : \mc S(\ZZ/\ell) \to \mc{PS}_i(\ell) := \PSp_{2g_i}(\ZZ/\ell)
		\]
		be the canonical projection. We prove the claim by induction on $n$. For $n = 1$ this follows from \cite[Lemma~6.1]{Jones_Gl2_reps}. Suppose now that $H \unlhd \mc S(\ZZ/\ell) = \prod_{i = 1}^n \Sp_{2g_i}(\ZZ/\ell)$ is a proper
		normal subgroup and that $H \not \subset \ker \varpi_i$ for any $i$. Let $G' := \prod_{i = 1}^{n - 1} \Sp_{2g_i}(\ZZ/\ell)$. Then, by the induction hypothesis $H$ surjects onto $G'$
		and onto $\Sp_{2g_n}(\ZZ/\ell)$. Therefore by Goursat's lemma  (cf. \cite[Lemma~(5.2.1)]{Ribet_Galois_action_on_division_points} for the version we use)
		\[
			H = G' \times_Q \Sp_{2g_n}(\ZZ/\ell)
		\]
		for some common quotient $Q$ of $G'$ and $\Sp_{2g_n}(\ZZ/\ell)$. Moreover, either
		$Q$ is trivial, or $Q$ surjects onto $\PSp_{2g_n}(\ZZ/\ell)$. However, in the latter case
		$H$ would not be normal, as one can see by conjugating by $\{ I \}^{n - 1} \times \Sp_{2g_n}(\ZZ/\ell)$.
		Hence $Q = \{ 1 \}$ and $H = \mc S(\ZZ/\ell)$.
		
		Finally, the only non-abelian composition factors of $\mc S(\ZZ/\ell')$
		are the groups $\mc{PS}_i(\ell')$ for $i = 1, \ldots, n$. But
		$\mc{PS}_i(\ell) \not \cong \mc{PS}_j(\ell')$ for $\ell \neq \ell'$,
		since finite simple groups of Lie type of different characteristic cannot be isomorphic (cf. \cite[Theorem~1.2]{Kantor_Seress_Large_element_orders}). This ends the proof of~(A1).\\
		\emph{Proof of (A0) and (A3):} For $n = 1$, this was proven in~\cite[Lemma 6.2]{Jones_Gl2_reps}. The proof follows, since $\mf s_{\ell}$ is the product $\prod_{i=1}^n \mf{sp}_{2g_i, \ell}$.
	\end{proof}
	\begin{Corollary} \label{cor:jones_2_18}
		Let $G(2) \subset \GSp^{\Delta}(\ZZ/2)$ be any subgroup and 
		let $G := \Phi_{\GSp^{\Delta}, 2}^{-1}(G(2))$. Suppose also that $\PSp_{2g_i}(\ZZ/\ell)$ is not a composition factor of $[G(2), G(2)]$ for every $i = 1, \ldots, n$ and $\ell > 2$. Let $H \subset G$ be any subgroup. Then $[G, G] = [H, H]$
		if and only if $[G(8), G(8)] = [H(8), H(8)]$ and for any prime $\ell > 2$ we have $[G(\ell), G(\ell)] = [H(\ell), H(\ell)]$.
	\end{Corollary}
	\begin{proof}
		This is basically \cite[Theorem 2.18]{Jones_Gl2_reps} for $\mc G = \GSp^{\Delta}$, $m = 2$, $\mc L = \{ 2 \}$ and $m_0 = 8$ (cf. \cite[formula~(22)]{Jones_Gl2_reps}). 
		Note however that we modified the assumptions from Jones' article:
		\begin{itemize}
			\item The homomorphisms $\varpi_i$ have different codomains, i.e. instead of considering homomorphisms $\varpi_i : \mc S(\ZZ/\ell) \to \mc{PS}(\ell)$, we consider
			homomorphisms $\varpi_i : \mc S(\ZZ/\ell) \to \mc{PS}_i(\ell)$.
			\item The condition (A1) of \cite{Jones_Gl2_reps} requires
			that $\mc{PS}_i(\ell)$ does not occur in $\mc S(\ZZ/\ell')$ for $\ell \neq \ell'$,
			i.e. that there don't exist subgroups $G_1, G_2$ of $\mc S(\ZZ/\ell')$ such that $G_2 \unlhd G_1$ and $G_1/G_2 \cong \mc{PS}_i(\ell)$. Instead, we assume that
			$\mc{PS}_i(\ell)$ is not a composition factor of $\mc S(\ZZ/\ell')$ for $\ell \neq \ell'$ and of $[G(2), G(2)]$.
			\item We do not require that $3 \in \mc L$. 
		\end{itemize}
		We justify now that these assumptions are obsolete in our case. One easily checks that replacing homomorphisms $\varpi_i : \mc S(\ZZ/\ell) \to \mc{PS}(\ell)$ by $\varpi_i : \mc S(\ZZ/\ell) \to \mc{PS}_i(\ell)$ does not affect the proof.
		
		Next, note that the assumption that $\mc{PS}_i(\ell)$ does not occur in $\mc S(\ZZ/\ell')$ for $\ell' \neq \ell$ is used in \cite{Jones_Gl2_reps} only for two statements:
		\begin{itemize}
			\item[(A1-1)] $\mc{PS}_i(\ell)$ is not a quotient of $\mc S(M')$ for $\ell \not \in \mc L$, $\ell \nmid M'$ (see proof of \cite[Lemma~4.9]{Jones_Gl2_reps}), 
			
			\item[(A1-2)] $\mc{PS}_i(\ell)$ is not a quotient of $[G(m_0), G(m_0)]$ for primes $\ell \not \in \mc L$ (see Subsection 4.4 ibid).
		\end{itemize}
		We justify now that~(A1-1) and~(A1-2) hold in our case. Indeed, in order to check~(A1-1) it suffices to show that $\mc{PS}_i(\ell)$
		is not a composition factor of $\mc S(M')$. If $M' = p_1^{\alpha_1} \cdot \ldots \cdot p_r^{\alpha_r}$, where $p_1, \ldots, p_r$ are primes distinct from $\ell$,
		then $\mc S(M') = \prod_{i = 1}^r \mc S(p_i^{\alpha_i})$. Thus
		the set of the composition factors of $\mc S(M')$ is the union of composition factors
		of $\mc S(p_i^{\alpha_i})$ for $i = 1, \ldots, n$. Moreover, $\ker(\mc S(p_i^{\alpha_i}) \to \mc S(p_i))$
		is a $p_i$-group (as it is a subgroup of $\ker(\Gl_{2g}(\ZZ/p_i^{\alpha_i}) \to \Gl_{2g}(\ZZ/p_i))$, cf.~\cite[ex. 2.23~(b)]{Silverman1994}) and hence it is solvable. Thus the non-abelian composition factors of $\mc S(p_i^{\alpha_i})$ are the non-abelian composition factors of $\mc S(p_i)$. Hence (A1-1) follows from our version of condition~(A1).
		Similarly, if $\mc{PS}_i(\ell)$ is not a composition factor
		of $[G(2), G(2)]$ for any $\ell \neq 2$, then it is also not a composition factor of $[G(8), G(8)]$, as $\ker([G(8), G(8)] \to [G(2), G(2)])$ is a $2$-group. Thus~(A1-2) is also true.\\
		
		Finally, for any prime $\ell$ denote by $e_{\ell}$ the smallest natural number such that for each
		closed subgroup $H \subset \mc S(\ZZ_{\ell})$ satisfying $H(\ell^{e_{\ell}}) = \mc S(\ZZ/\ell^{e_{\ell}})$
		one has $H = \mc S(\ZZ_{\ell})$. The assumption that $3 \in \mc L$ is used only in \cite[Remark 2.16]{Jones_Gl2_reps},
		in order to show that $e_3 \le 2$. In our case $e_{\ell} = 1$ for every~$\ell$,
		see Lemma~\ref{lem:product_of_sp}~(1) below. 
	\end{proof}
	\begin{Lemma} \label{lem:product_of_sp}
	\begin{enumerate}[(1)]
		\item[] 
		\item Let $\ell$ be an arbitrary prime. If $H \le \prod_{i = 1}^n \Sp_{2g_i}(\ZZ_{\ell})$ is a closed subgroup and $H(\ell) = \prod_{i = 1}^n \Sp_{2g_i}(\ZZ/\ell)$, then $H = \prod_{i = 1}^n \Sp_{2g_i}(\ZZ_{\ell})$.
		
		\item Fix $m \ge 1$. If $H \le \prod_{i = 1}^n \mc S \mathscr S_i(2^m)$ is a subgroup
		and $H(2) = \prod_{i = 1}^n S_i$, then $H = \prod_{i = 1}^n \mc S \mathscr S_i(2^m)$.
	\end{enumerate}
		
	\end{Lemma}
	\begin{proof}
		(1) We proceed by induction on $n$. The case $n = 1$ follows by \cite[Theorem 1]{Landesman_Lifting_Symplectic}.
		Suppose that the statement holds for $n - 1$. Denote $G' := \prod_{i = 1}^{n - 1} \Sp_{2g_i}(\ZZ_{\ell})$. Let $H_1 := \im (H \to G')$.
		Then $H_1(\ell) = G'(\ell)$ and hence by the induction hypothesis,
		$H_1 = G'$. Analogously, $H$ surjects onto
		$\Sp_{2g_n}(\ZZ_{\ell})$. Let $N := \ker(H \to \Sp_{2g_n}(\ZZ_{\ell})) \unlhd G'$. Then,
		since $H(\ell) = \prod_{i = 1}^n \Sp_{2g_i}(\ZZ/\ell)$, we have $N(\ell) = G'(\ell)$. Therefore
		by the induction hypothesis we have $N = G'$. One concludes
		using Goursat's lemma.\\
		(2) The case $n = 1$ is~\cite[Theorem~1.1(a)]{Yelton_Lifting_images}. One proceeds by induction, analogously as in~(1).
	\end{proof}	
	Before the proof we need one more algebraic lemma.
	\begin{Lemma} \label{lem:lem_subgroup_surjects_i_j}
		Suppose that $\ell > 2$ is a prime and that $g_1, \ldots, g_n \ge 2$. If $H \le \prod_{i = 1}^n \Sp_{2 g_i}(\ZZ/\ell)$
		is a subgroup surjecting onto $\Sp_{2 g_i}(\ZZ/\ell) \times \Sp_{2 g_j}(\ZZ/\ell)$ for every $1 \le i < j \le n$,
		then $H = \prod_{i = 1}^n \Sp_{2 g_i}(\ZZ/\ell)$.
	\end{Lemma}
	\begin{proof}
		We prove this by induction on $n$. For $n = 2$ this is immediate. Suppose now that
		the result holds for $n-1$. Let $G' := \prod_{i = 1}^{n - 1} \Sp_{2 g_i}(\ZZ/\ell)$. If $H$ satisfies the hypothesis of the statement, then by the induction hypothesis, $H$ surjects onto $G'$. Therefore we may use Goursat's lemma
		for $H \le G' \times \Sp_{2 g_n}(\ZZ/\ell)$. Let
		\[
		N := \ker(H \to \Sp_{2 g_n}(\ZZ/\ell)) \unlhd G'.
		\]
		Suppose by way of contradiction that $N \neq G'$. Then, by condition (A1), $N \le \ker \varpi_j$ 
		for some $j$. On the other hand, $H$ surjects onto $\Sp_{2 g_j}(\ZZ/\ell) \times \Sp_{2 g_n}(\ZZ/\ell)$.
		Therefore for every $s \in \Sp_{2 g_j}(\ZZ/\ell)$ there exists $h \in H$, which maps
		to $(s, I)$ in $\Sp_{2 g_j}(\ZZ/\ell) \times \Sp_{2 g_n}(\ZZ/\ell)$. By definition $h \in N$
		and $h$ maps to $s$ in $\Sp_{2 g_j}(\ZZ/\ell)$. Thus $N$ surjects onto $\Sp_{2 g_j}(\ZZ/\ell)$.
		The contradiction means that $N = G'$ and hence $H = \prod_{i = 1}^n \Sp_{2 g_i}(\ZZ/\ell)$ by Goursat's lemma.
	\end{proof}
	\begin{proof}[Proof of Proposition~\ref{prop:criterion_based_on_jones}]
		Let $\mc G$ and $\delta$ be as in Lemma~\ref{lem:A0-A3}. 
		Denote $G = \Phi_{\mc G, 2}^{-1}(\SS_1 \times \ldots \times \SS_n)$ and $H = \im \rho_J$.
		Fix a prime $\ell > 2$.
		Since $\lambda \circ \rho_J$ is the cyclotomic character, we have $\mc S (H(\ell)) = \Gal(\QQ(J[\ell])/\QQ(\zeta_{\ell}))$. 
		Note that by Lemma~\ref{lem:image_on_nondiv_by_Mi}, since $M_i \nmid \ell$, $\Gal(\QQ(J_i[\ell])/\QQ(\zeta_{\ell})) \cong \ms S_i(\ell) \cong \Sp_{2g_i}(\ZZ/\ell)$. Moreover, since $\QQ(J_1[\ell]), \ldots, \QQ(J_n[\ell])$ are pairwise linearly disjoint over~$\QQ(\zeta_{\ell})$, $\Gal(\QQ(J[\ell])/\QQ(\zeta_{\ell}))$ surjects onto
		\[
			\Gal(\QQ(J_i[\ell], J_j[\ell])/\QQ(\zeta_{\ell})) \cong \Sp_{2 g_i}(\ZZ/\ell) \times \Sp_{2 g_j}(\ZZ/\ell).
		\]
		for every $i \neq j$. Hence by Lemma~\ref{lem:lem_subgroup_surjects_i_j}
		\[
			\mc S(H(\ell)) = \prod_{i=1}^n \Sp_{2 g_i}(\ZZ/\ell).
		\]
		Therefore, since
		\[
			\mc S(H(\ell)) \le H(\ell) = \Gal(\QQ(J[\ell])/\QQ) \le \mc G(\ZZ/\ell),
		\]
		we have
		\[
			\prod_{i = 1}^n \Sp_{2g_i}(\ZZ/\ell) \le [\mc S (H(\ell)), \mc S (H(\ell))] \le [H(\ell), H(\ell)] \le [\mc G(\ZZ/\ell), \mc G(\ZZ/\ell)]
			= \prod_{i = 1}^n \Sp_{2g_i}(\ZZ/\ell).
		\]
		This yields that $[H(\ell), H(\ell)] = [G(\ell), G(\ell)]$.
		
		Moreover, since the fields $\QQ(J_1[2]), \ldots, \QQ(J_n[2])$ are linearly disjoint over $\QQ$, we have
		\[
			\mc S (H(2)) = \Gal(\QQ(J[2])/\QQ) = \prod_{i = 1}^n \SS_i.
		\]
		This implies by Lemma~\ref{lem:product_of_sp}~(2) that $\mc S (H(8)) \cong \prod_{i = 1}^n \mc S \mathscr S_i(8)$.
		Analogously as before, one obtains $[H(8), H(8)] = [G(8), G(8)]$.
		Moreover, $\PSp_{2 g_i}(\ell)$ is not a composition factor of $[G(2), G(2)] = \prod_{i = 1}^n \AA_i$ for any $i = 1, \ldots, n$ and any $\ell > 2$. Indeed,
		the only composition factors of $\prod_{i = 1}^n \AA_i$ are the simple groups $\AA_1, \ldots, \AA_n$ and
		they are not isomorphic to $\PSp_{2 g_i}(\ell)$ (cf. list of exceptional isomorphisms of simple groups in \cite[\S 3.5]{Conway_Curtis_ATLAS}). Therefore
		by Corollary~\ref{cor:jones_2_18} we obtain $[G, G] = [H, H]$.
		Thus by Lemma~\ref{lem:lower_bound_index} and the third isomorphism theorem
		\begin{align*}
			[G:H] &= [\mc S G : [G, G]] = [\mc S(G(2)) : [G(2), G(2)]]\\
			&= \left[\prod_{i = 1}^n \SS_i : \left[\prod_{i = 1}^n \SS_i, \prod_{i = 1}^n \SS_i \right] \right] = \left[\prod_{i = 1}^n \SS_i : \prod_{i = 1}^n \AA_i \right]\\
			&= 2^n. \qedhere
		\end{align*}
	\end{proof}
	\section{$2$-division fields} \label{sec:8-torsion}
	\noindent Keep the assumptions of Theorem~\ref{thm:main_theorem}. In this section we prove that
	the $2$-division fields are linearly disjoint over $\QQ$. To this end we need to show that
	for any $1 \le i \le n$
	\begin{equation} \label{eqn:intersection_8}
		\QQ(J_i[2]) \cap \QQ(J'_i[2]) = \QQ,
	\end{equation}
	where $J'_i = \prod_{j \neq i} J_j$. Denote the left hand side by $L_2$. Then
	\[
		N_i := \Gal(\QQ(J_i[2])/L_2) \le \Gal(\QQ(J_i[2])/\QQ) = \SS_i.
	\]
	Observe that $\QQ(J_i[2])$ and $\QQ(J_i'[2])$ are Galois over $\QQ$,
	thus their intersection $L_2$ is also Galois over $\QQ$ (cf. \cite[Proposition~21~(1)]{DummitFoote2004}).
	This implies that $N_i = \Gal(\QQ(J_i[2])/L_2)$ is a normal subgroup of $\Gal(\QQ(J_i[2])/\QQ) = \SS_i$.
	Denote by $\AA_i \subset \SS_i$ the alternating group on $\deg f_i$ elements. Note that $N_i \not \subset \AA_i$. Indeed, otherwise we would have
	\[
	\QQ(\sqrt{\Delta_i}) = \QQ(J_i[2])^{\AA_i} \subset \QQ(J_i[2])^{N_i} = L_2 \subset \QQ(J'_i[2])
	\]
	and $\QQ(J'_i[2])/\QQ$ would be ramified over $\ell_i$. But this yields a contradiction,
	since $\QQ(J'_i[2])$ is the compositum of the splitting fields of the polynomials $\{ f_j : j =1, \ldots, n, j \neq i \}$,
	whose discriminants are not divisible by $\ell_i$.
	Thus $N_i$ is a normal subgroup of $\SS_i$ not contained in $\AA_i$. But, since $g_i \ge 2$, 
	$\AA_i$ is simple, which easily implies that $N_i = \SS_i$ and $L_2 = \QQ$. This proves~\eqref{eqn:intersection_8}.

\section{$\ell$-division fields for $\ell > 2$} \label{sec:p-torsion}
	In this section we finish the proof of Theorem~\ref{thm:main_theorem}. We show first that for any prime $\ell > 2$ and fixed $1 \le i < j \le n$
	\begin{equation} \label{eqn:intersection_p}
		\QQ(J_i[\ell]) \cap \QQ(J_j[\ell]) = \QQ(\zeta_{\ell}).
	\end{equation}
	Denote the left hand side by $L_{\ell}$ and let $N_i := \Gal(\QQ(J_i[\ell])/L_{\ell})$.
	Before the proof recall that the condition $\ord_{\ell_i}(\Delta_i) = 1$ implies that $J_i$ has a semistable reduction of toric rank one over $\QQ_{\ell_i}$
	(cf. Corollary~\ref{cor:toric_rk_one}).\\
	
	Suppose that $\ell \neq \ell_i$ (otherwise we may switch $i$ and $j$).
	Let $\mf p$ denote an arbitrary prime of $\QQ(J_i[\ell])$ over $\ell_i$
	and let $I_{\mf p}$ denote the inertia subgroup of $\mf p$ in the extension $\QQ(J_i[\ell])/\QQ$.
	Note that $L_{\ell} \subset \QQ(J_j[\ell])$. Hence, by the N\'{e}ron-Ogg-Shafarevich criterion (cf. \cite[Theorem 1]{Serre_Tate_Good_reduction}),
	$L_{\ell}/\QQ$ is unramified over $\ell_i$. This implies that $I_{\mf p} \subset N_i$.
	Similarly as in Section~\ref{sec:8-torsion}, one shows that $N_i$ is a normal subgroup of $\Gal(\QQ(J_i[\ell])/\QQ(\zeta_{\ell})) \cong \Sp_{2g_i}(\ZZ/\ell)$ (the last isomorphism follows from
	the fact that $J_i$ has maximal Galois image). On the other hand, $\PSp_{2g_i}(\ZZ/\ell)$ is a simple group and hence
	$N_i = \{ I \}, \{ \pm I \}$ or $\Sp_{2g_i}(\ZZ/\ell)$. 
	By the N\'{e}ron-Ogg-Shafarevich criterion, $I_{\mf p} \neq \{ I \}$. 
	Suppose by way of contradiction that $I_{\mf p} = \{ \pm I \}$. Let
	$J_{i, \ell_i}$ be the special fiber of the N\'{e}ron model of $J_i/\QQ$ over $\ell_i$.
	Then by \cite[Lemma~2]{Serre_Tate_Good_reduction}
	\[
		J_{i, \ell_i}[\ell] \cong J_i[\ell]^{I_{\mf p}} \cong 0,
	\]
	since the fixed subspace of the action of $-I$ on an arbitrary $\ZZ/\ell$-vector space is trivial. But the $\ell$-torsion of $J_{1, \ell_i}$ may vanish if and only if
	the reduction is totally unipotent, which leads to a contradiction. Therefore $N_i = \Sp_{2g_i}(\ZZ/\ell)$ and~\eqref{eqn:intersection_p} holds. The equalities~\eqref{eqn:intersection_8} and~\eqref{eqn:intersection_p}
	show that the assumptions of Proposition~\ref{prop:criterion_based_on_jones} are satisfied. Thus~$J$ has
	maximal Galois image, which ends the proof of Theorem~\ref{thm:main_theorem}.

	\section{Proof of Theorem~\ref{thm:division_fields}} \label{sec:corollary}
	Keep the notation of Theorem~\ref{thm:division_fields} and
	denote $H := \im \rho_J$, $G:= \Phi_{\mc G, 2}^{-1}(\SS_1 \times \ldots \times \SS_n)$, $G_i = \Phi_{\mc G, 2}^{-1}(\SS_i)$, $H_i := \im \rho_{J_i}$. Let $\QQ(\zeta_{\infty})$ denote the compositum of all cyclotomic fields and
	let $\QQ(A[\infty])$ be the compositum of all division fields of an abelian variety $A$ defined over $\QQ$.
	Note that since $J_i$ has maximal Galois image, Lemma~\ref{lem:lower_bound_index} implies that $[G_i, G_i] = [H_i, H_i] = \mc SH_i$.
	Moreover, by \cite[Lemma~1, p.~174]{Lang_Trotter_Frobenius_distributions}, $[G, G] = \prod_{i = 1}^n [G_i, G_i]$. Hence
	\begin{align*}
		\Gal(\QQ(J[\infty])/\QQ(\zeta_{\infty})) &= \mc S H = [G, G] = \prod_{i = 1}^n [G_i, G_i]\\
		&= \prod_{i = 1}^n \mc S H_i = \prod_{i = 1}^n \Gal(\QQ(J_i[\infty])/\QQ(\zeta_{\infty})).
	\end{align*}
	Therefore the fields $\QQ(J_1[\infty]), \ldots, \QQ(J_n[\infty])$ are linearly disjoint over $\QQ(\zeta_{\infty})$. Thus
	\begin{equation} \label{eqn:intersection_contained_in_Qinfty}
		\QQ(J_{\ms A}[m_1]) \cap \QQ(J_{\ms B}[m_2]) \subset \QQ(\zeta_{\infty}).
	\end{equation}
	\begin{Lemma} \label{lem:maximal_abelian}
		Keep the assumptions of Theorem~\ref{thm:division_fields}.
		Then for any $m \in \NN$
		\[
			\QQ(J[m]) \cap \QQ(\zeta_{\infty}) = \QQ \left(\zeta_m, \sqrt{\Delta_1^{m + 1}}, \ldots, \sqrt{\Delta_n^{m + 1}} \right).
		\]
	\end{Lemma}
	\begin{proof}
		Keep the above notation.
		Recall that since $J$ has maximal image, we have $[H, H] = [G, G]$ by Lemma~\ref{lem:lower_bound_index}.
		Thus
		\begin{align*}
			\Gal(\QQ(J[m])/\QQ(J[m]) \cap \QQ(\zeta_{\infty})) &= [H(m), H(m)] = [G(m), G(m)]\\
			&= \Phi^{-1}_{\mc S, 2}(\AA_1 \times \ldots \times \AA_n)(m)\\
			&= \mc S (H(m)) \cap \Phi^{-1}_{\mc G, 2}(\AA_1 \times \ldots \times \AA_n)(m)
		\end{align*}
		(the first equality follows from the fact that $\QQ(J[m]) \cap \QQ(\zeta_{\infty})$ is the maximal subfield of $\QQ(J[m])$ which is abelian over~$\QQ$).
		Observe now that since $\lambda \circ \rho_J$ equals the cyclotomic character, the subgroup $\mc S (H(m)) \subset H(m)$ corresponds to 
		the subfield $\QQ(\zeta_m)$ of $\QQ(J[m])$. Moreover, if $m$ is odd, then by the Chinese Remainder Theorem,
		$\Phi^{-1}_{\mc G, 2}(\AA_1 \times \ldots \times \AA_n)$ surjects onto $H(m)$. Thus $\Phi^{-1}_{\mc G, 2}(\AA_1 \times \ldots \times \AA_n)(m) = H(m)$
		corresponds to the subfield $\QQ$ of $\QQ(J[m])$. If $m$ is even, then for any $\sigma \in H(m)$,
		we have
		\begin{align*}
			\sigma \in \Phi^{-1}_{\mc G, 2}(\AA_1 \times \ldots \times \AA_n)(m) \quad &\Leftrightarrow \quad \sigma|_{\QQ(J[2])} \in \AA_1 \times \ldots \times \AA_n\\
			&\Leftrightarrow \quad \sigma(\Delta_i) = \Delta_i \quad \textrm{ for } i = 1, \ldots, n\\
			&\Leftrightarrow \quad \sigma \in \Gal \left(\QQ(J[m])/\QQ\left(\sqrt{\Delta_1}, \ldots, \sqrt{\Delta_n} \right) \right).
		\end{align*}
		This means that $\Phi^{-1}_{\mc G, 2}(\AA_1 \times \ldots \times \AA_n)(m)$ corresponds to the subfield $\QQ\left(\sqrt{\Delta_1}, \ldots, \sqrt{\Delta_n} \right)$.
		Therefore, for an arbitrary $m \in \NN$ the subgroup $\Phi^{-1}_{\mc G, 2}(\AA_1 \times \ldots \times \AA_n)(m) \subset H(m)$ corresponds
		to the subfield $\QQ \left(\sqrt{\Delta_1^{m + 1}}, \ldots, \sqrt{\Delta_n^{m + 1}} \right)$.
		Hence the intersection of subgroups $\mc S (H(m)) \cap \Phi^{-1}_{\mc G, 2}(\AA_1 \times \ldots \times \AA_n)(m)$
		corresponds to the compositum $\QQ(\zeta_m) \cdot \QQ \left(\sqrt{\Delta_1^{m + 1}}, \ldots, \sqrt{\Delta_n^{m + 1}} \right)$, which ends the proof.
	\end{proof}
	In order to finish the proof of Theorem~\ref{thm:division_fields}, note that by~\eqref{eqn:intersection_contained_in_Qinfty} and by Lemma~\ref{lem:maximal_abelian}
	\begin{align*}
		\QQ(J_{\ms A}[m_1]) \cap \QQ(J_{\ms B}[m_2]) &= (\QQ(J_{\ms A}[m_1]) \cap \QQ(\zeta_{\infty})) \cap (\QQ(J_{\ms B}[m_2]) \cap \QQ(\zeta_{\infty}))\\
		&= \QQ(\zeta_{m_1}) \left(\sqrt{\Delta_i^{m_1 + 1}} : i \in \ms A \right) \cap \QQ(\zeta_{m_2})\left(\sqrt{\Delta_i^{m_2 + 1}} : i \in \ms B \right).
	\end{align*}
	\section{Example~\ref{ex:family}} \label{sec:example}
	Keep the notation of Example~\ref{ex:family}. First, we note that for every $t \in \ZZ$, the Jacobian $J_t$ has maximal Galois image. Indeed, by \cite[\S 8]{Anni_Dokchitser_Constructing_hyperelliptic} one sees that $\Sp_{12}(\ZZ/\ell) \subset \im \rho_{J_t}(\ell)$ for every prime $\ell > 2$, and that $\im \rho_{J_t}(2) = S_{14}$ (the symmetric group on fourteen elements).
	Therefore, by \cite[Theorem 1.1(a)]{Yelton_Lifting_images}, $\Phi^{-1}_{\Sp_{12}, 2}(S_{14})(2^{\infty}) \subset \im \rho_{J_t, 2}$.
	The surjectivity of the multiplier map implies that $\Phi^{-1}_{\GSp_{12}, 2}(S_{14})(2^{\infty}) \subset \im \rho_{J_t, 2}$.
	We use \cite[Lemma 4.3]{Landesman_Hyperelliptic_curves_maximal} to conclude that $\rho_{J_t}$ has maximal Galois image.
	
	 Let $\Delta(t) \in \ZZ[t]$ denote the discriminant of the polynomial $F_t$ and let $d \in \ZZ$ be the discriminant
	of $\Delta(t)$. One checks, using SageMath, that $d \neq 0$. We construct inductively an increasing sequence $\{ t_i \}_{i = 1}^{\infty}$ of integers
	and a sequence $\{\ell_i\}_{i = 1}^{\infty}$ of primes such that $\ord_{\ell_i}(\Delta(t_i)) = 1$ and $\ord_{\ell_j}(\Delta(t_i)) = 0$
	for $i \neq j$. Note that we can take $t_1 = 0$ and $\ell_1 = 89$. Suppose now that we constructed $t_1, \ldots, t_{n - 1}$
	and $\ell_1, \ldots, \ell_{n-1}$. Using a theorem of Schur (cf. \cite[\S 20.1, p.~1]{SchurII}) we can find a prime $\ell_n \not \in \{ \ell_1, \ldots, \ell_{n-1} \}$, $\ell_n \nmid d \cdot \Delta(0)$ and $t_n \in \ZZ$ such that $\ell_n|\Delta(t_n)$. If $\ell_n^2 \nmid \Delta(t_n)$ then we are done. Otherwise,
	we may replace $t_n$ by $t_n + \ell_n$. Indeed, suppose to the contrary that $\ell_n^2 | \Delta(t_n), \Delta(t_n + \ell_n)$.
	Then
	\[
		0 \equiv \Delta(t_n + \ell_n) - \Delta(t_n) \equiv \ell_n \cdot \Delta'(t_n) \pmod{\ell_n^2},
	\] 
	where $\Delta'(t)$ denotes the derivative of the polynomial $\Delta(t)$.
	Thus $\ell_n | \Delta(t_n), \Delta'(t_n)$. In other words, $t_n$ is a double root of $\Delta(t) \pmod{\ell_n}$
	and $d \equiv 0 \pmod{\ell_n}$. This yields a contradiction, since $\ell_n \nmid d$. By applying the Chinese Remainder Theorem, we may also enlarge $t_n$ to ensure that $\ell_i \nmid \Delta(t_n)$ for $i < n$. Therefore $J := J_{t_1} \times \ldots \times J_{t_n}$ satisfies the hypotheses of Theorem~\ref{thm:main_theorem}. This ends the induction.\\
	
	Finally, the following code in SageMath
	\begin{python}
		R.<t> = PolynomialRing(QQ)
		Rx.<x> = PolynomialRing(R)
		N = 2201590757511816436065484800
		f = x^(14)+1122976550518058592759939074*x^(13)
		+10247323490706358348644352 * x^(12)+1120184609916242124087443456*x^(11)
		+186398290364786000921886720*x^(10)+1685990245699349559300014080*x^9
		+387529952672653585935499264*x^8+1422826957983635547417870336*x^7
		+585983998625429997308035072*x^6+607434202225985243206107136*x^5
		+1820210247550502007557029888*x^4+533014336994715937945092096*x^3
		+595803405154942945879752704*x^2+1276845913825955586899050496*x
		+1323672381818030813822668800+t*N
		delta = f.discriminant()
		T = [0, 2, 5, 7, 9, 10, 13, 14, 17, 21]
		ell = [89, 983, 839, 43, 31, 167, 103, 40829, 653, 11969]
		for i in range(10):
			for j in range(10):
				if i == j:
					print(delta(t = T[i]).valuation(ell[j]) == 1)
				else:
					print(delta(t = T[i]).valuation(ell[j]) == 0)
	\end{python}
	verifies that the primes
	\[
		(\ell_1, \ldots, \ell_{10}) := (89, 983, 839, 43, 31, 167, 103, 40829, 653, 11969)
	\]
	satisfy the required assumptions for $(t_1, \ldots, t_{10}) := (0, 2, 5, 7, 9, 10, 13, 14, 17, 21)$.
	\bibliography{bibliografia}
\end{document}